\documentclass{siamltex}
\usepackage{amssymb}
\usepackage{bm}
\usepackage{graphicx}
\usepackage{graphics}
\usepackage{caption2}
\usepackage{psfrag}
\usepackage{arydshln}
\usepackage{amsmath}
\usepackage{amscd}
\usepackage{amsfonts}
\usepackage{booktabs}
\usepackage{float}
\usepackage{verbatim}
\usepackage{latexsym}
\usepackage{color}
\usepackage{multirow}
\usepackage{lscape}
\usepackage{arydshln}
\usepackage{epstopdf}
\usepackage{extarrows}
\newcommand{\om}{\Omega}

\def\C{\mathbb{C}}
\newcommand{\p}{\partial}

\newcommand{\ds}{\displaystyle}
\newcommand{\f}{\frac}

\begin{document}


\title{A multiscale Abel kernel and application in viscoelastic problem}

\author{
Wenlin Qiu\thanks{School of Mathematics, Shandong University, Jinan 250100, China. (Email: wlqiu@sdu.edu.cn)}
\and
Tao Guo\thanks{School of Mathematics and Statistics, Hunan Normal University, Changsha, Hunan 410081, China. (Email: guotao6613@163.com)}
\and
Yiqun Li\thanks{School of Mathematics and Statistics, Wuhan University, Wuhan 430072, China. (Email: YiqunLi24@outlook.com)}
\and
Xu Guo\thanks{Geotechnical and Structural Engineering Center, Shandong University, Jinan 250061, China. (Email: guoxu@sdu.edu.cn)}
\and
Xiangcheng Zheng\thanks{Corresponding author. School of Mathematics, Shandong University, Jinan 250100, China. (Email: xzheng@sdu.edu.cn)}
}
\maketitle

\begin{abstract}
We consider the variable-exponent Abel kernel and demonstrate its multiscale nature in modeling crossover dynamics from the initial quasi-exponential behavior to long-term power-law behavior. Then we apply this to an integro-differential equation modeling, e.g. mechanical vibration of viscoelastic materials with changing material properties. We apply the Crank-Nicolson method and the linear interpolation quadrature to design a temporal second-order scheme, and develop a framework of exponentially weighted energy argument in error estimate to account for the non-positivity and non-monotonicity of the multiscale kernel. Numerical experiments are carried out to substantiate the theoretical findings and the crossover dynamics of the model.
\end{abstract}



\pagestyle{myheadings}
\thispagestyle{plain}

\markboth{Qiu, Guo, Li, Guo and Zheng}{A multiscale power-law kernel and its application}

\section{Introduction}
\subsection{A multiscale kernel}\label{Sect:k}

We consider the following Abel kernel with the variable exponent $0 < \alpha(t) \leq 1$ \cite{Fan,SunCha}
\begin{equation}\label{k}
\ds k(t) : =  \frac{t^{\alpha(t)-1}}{\Gamma(\alpha(t))}
\end{equation}
where $\Gamma(\cdot)$ is the Euler's Gamma function. Compared with the commonly-used constant-exponent kernel ($\alpha(t)\equiv\alpha$ for some $0<\alpha<1$), which models the power-law phenomena in various applications \cite{Deng,Hao,Jiang}, the variable-exponent kernel \eqref{k} could model more complex phenomena with multiple scales. A typical example is the normal-anomalous diffusion of the particles in worm-like micellar solutions observed in experiments in \cite{Jeo}, which is a crossover dynamics with multiple diffusion scales in time such that a constant-exponent (i.e. single-scale) power-law kernel could not accommodate the transition between multiscale diffusion processes. Instead, it is shown in \cite{ZheLiQiu} that the variable-exponent kernel (\ref{k}) with $\alpha(0)=1$ provides a local modification of subdiffusion that could characterize the normal-anomalous diffusion process.

The success of describing the normal-anomalous diffusion process in \cite{ZheLiQiu} essentially relies on the multiscale feature of the variable-exponent kernel (\ref{k}) with $\alpha(0)=1$. To give a better understanding of its multiscale nature, we first select $\alpha(t)=0.9+0.1e^{-0.1t}$ as an example, which satisfies the following asymptotics
 $$\alpha(t)\sim\alpha(0)+\alpha'(0)t\text{ for }t \text{ small enough and }\alpha(t)\sim 0.9 \text{ for }t \text{ large enough}. $$
 We employ these to compare the variable-exponent Abel kernel (\ref{k}) with its local asymptotics at $t=0$ and $t=\infty$
 \begin{equation}\label{kzz1}
 k_0(t):=\frac{t^{(\alpha(0)+\alpha'(0)t)-1}}{\Gamma(\alpha(0))}= e^{\alpha'(0)t\ln t},~~k_{\infty}(t):=\frac{t^{0.9-1}}{\Gamma(0.9)}=\frac{t^{-0.1}}{\Gamma(0.9)}
 \end{equation}
 in Fig.~\ref{Fig1}(left), from which we observe that the variable-exponent kernel exhibits a crossover dynamics  from the initial quasi-exponential behavior (modeled by $k_0$) to the long-term power-law decay (modeled by $k_\infty$), indicating the multiscale feature of  \eqref{k}. By this means, the variable-exponent kernel   $k$ not only eliminates the initial rapid change and thus the initial singularity of the constant-exponent kernel $k_\infty$, but captures its long-time power-law behavior that $k_0$ does not exhibit. For these reasons, we call (\ref{k}) the {\it multiscale kernel}.

\begin{figure}[h]
\setlength{\abovecaptionskip}{0pt}
\centering
\includegraphics[width=2.5in,height=2in]{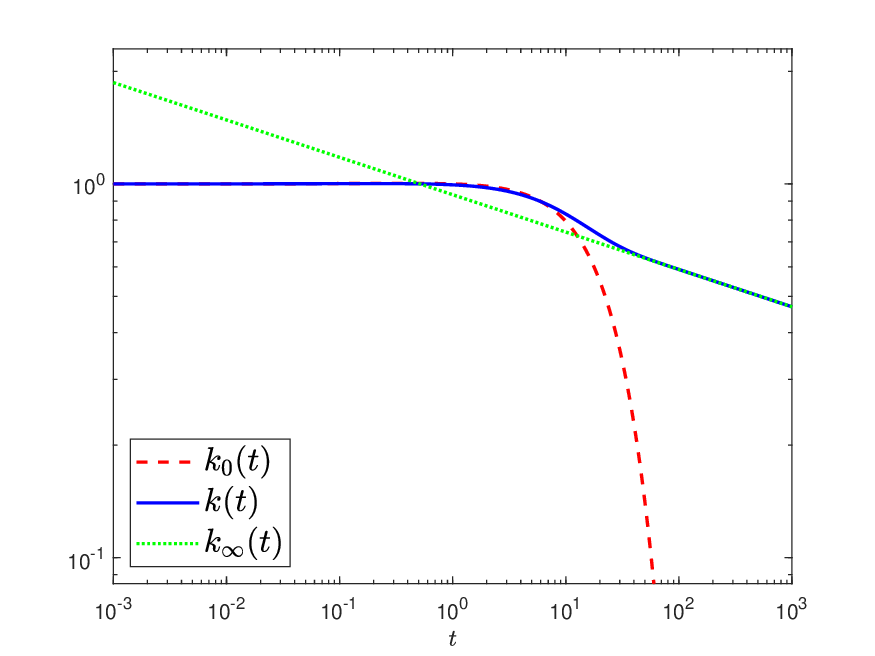}
\includegraphics[width=2.5in,height=2in]{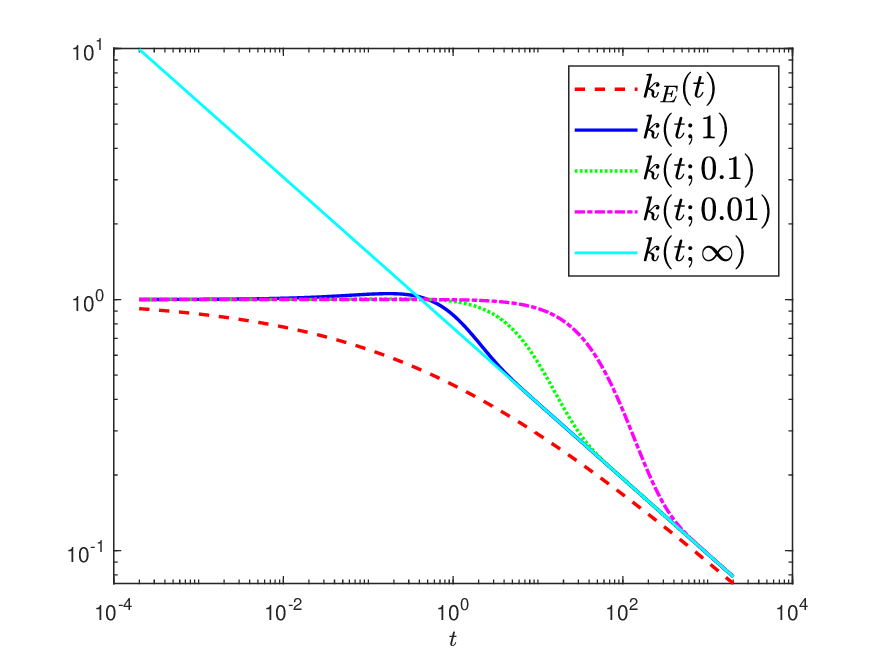}
\caption{{\footnotesize Log-log plots of (left) $k(t)$ in (\ref{k}) with $\alpha(t) = 0.9 + 0.1 e^{-0.1t} $ and its asymptotics in (\ref{kzz1}) and (right) $k(t;a)$ with $\alpha(t;a) = 0.7 + 0.3 e^{-at} $ for different $a$ and the Mittag-Leffler kernel $k_E(t)$.}}
\label{Fig1}\vspace{-0.1in}
\end{figure}

It is worth mentioning that similar multiscale features could also be realized by the well-known Mittag–Leffler kernel
 $$k_E(t):=E_{\beta,1}(-t^\beta),~~E_{\beta,1}(z):=\sum_{i=0}^\infty \frac{z^i}{\Gamma(\beta i+1)},~~z\in\mathbb R$$
  for some $0<\beta<1$. It is demonstrated in \cite[Appendix B]{Bon} and \cite{MukPar} that the Mittag–Leffler kernel behaves like a stretched exponential function, i.e. $e^{\frac{-t^\beta}{\Gamma(1+\beta)}}$, at short times with the stretching exponent $\beta$, and then exhibits a power-law decay like $\frac{t^{-\beta}}{\Gamma(1-\beta)}$ as $t$ tends to infinity. Nevertheless, the proposed muitiscale kernel (\ref{k}) has the following advantages in comparison with the Mittag–Leffler kernel:
\begin{itemize}
\item[$\bullet$]  As the Mittag-Leffler function is typically defined by an infinite series \cite{Jinbook}, it is difficult to evaluate the Mittag-Leffler kernel and the obtained value may not be accurate. In contrast, the multiscale kernel \eqref{k} could be simply evaluated with a very high accuracy.

\item[$\bullet$] The Mittag-Leffler kernel gets close to the power law only if $t$ tends to infinity, while the power law behaviors appear within finite times in most applications. Instead, the variable exponent provides a great flexibility in adjusting the properties of the kernel, which could account for more complicated scenarios.
\end{itemize}
To better illustrate the second statement,  we present the curves of the Mittag-leffler kernel $k_E(t)$ with $\beta=0.3$ and a parameterized multiscale kernel $$k(t;a)=\frac{t^{\alpha(t;a)-1}}{\Gamma(\alpha(t;a))},~~\alpha(t;a) = 0.7 + 0.3 e^{-at} $$
 under different parameters $a$ in Fig.~\ref{Fig1}(right).  We  observe that the Mittag-leffler kernel $k_E(t)$ approaches the  power-law decay (modeled by $k(t;\infty)=\frac{t^{-0.3}}{\Gamma(0.7)}$) only for very large $t$, while the multiscale kernel could tend to the power-law decay at different times by adjusting the parameter $a$.
  Consequently, the above discussions demonstrate the flexibility and the novelty of the multiscale kernel (\ref{k}) and motivate the application of this kernel in practical models.

\subsection{Modeling issues} \label{Sect:model}
We apply the multiscale kernel (\ref{k}) in the following parabolic integro-differential equation (PIDE), which attracts increasing attentions in viscoelastic material modeling \cite{Dan,Friedman,Heard}
     \begin{equation}\label{eq1.1}
     \begin{split}
          \frac{\partial u}{\partial t}(\bm x,t) - \mu \Delta u(\bm x,t) - \zeta I^{(\alpha(t))} \Delta u(\bm x,t) = f(\bm x,t),   \quad  \bm x \in \Omega,  \quad   0< t \leq T
     \end{split}
     \end{equation}
     with the initial conditions
     \begin{equation}\label{eq1.2}
       u(\bm x,0)=u_0(\bm x), \quad  \bm x \in \Omega \cup\partial\Omega
     \end{equation}
     and boundary conditions
     \begin{equation}\label{eq1.3}
     \begin{split}
     u(\bm x,t)= 0,  \quad   \bm x \in \partial\Omega, \quad 0 < t\leq T.
     \end{split}
     \end{equation}
  Here $\Omega\subset \mathbb R^d$ ($1\leq d\leq 3$) is an open bounded smooth domain with boundary $\p\Omega$, $T>0,$ the source term $f(\bm x,t)$ and the initial value $u_0(\bm x)$ are given functions, $\mu>0$ represents the viscosity coefficient, $\zeta >0$, and the variable-exponent integral operator $I^{(\alpha(t))}$   is defined as follows \cite{LorHar,Orosco}
     \begin{equation}\label{eq1.4}
     \begin{split}
          I^{(\alpha(t))}\varphi(t)=(k* \varphi)(t) :=\int_{0}^{t}k(t-s)\varphi(s)ds.
     \end{split}
     \end{equation}

In \eqref{eq1.1},  the variable exponent determined by the characteristic fractal dimensions of the microstructures  characterizes the structure changes  of viscoelastic systems due to long-term cyclic loads, which in turn propagate to macro scales that eventually result in material failure \cite{LiWanZhe,MeeSik}.
Furthermore,  it is validated in \cite{MenYin} that the variable-exponent model may  predict  the compression deformation of amorphous glassy polymers with higher accuracy and fewer parameters, which motivates the application of the variable exponent in \eqref{eq1.1} to depict the changing physical properties of viscoelastic materials.

To demonstrate the impact of the multiscale kernel on solutions, we set $\Omega=(0,1)$,  $\mu=\zeta=1$, $f \equiv 1$ and $u_0 = \sin(\pi x)$ in (\ref{eq1.1})-(\ref{eq1.3}) and numerically compute $\p_t u(0.5,t)$ over a short time period $[0, 0.1]$ under a multiscale kernel and a constant-exponent kernel, respectively, in Fig \ref{Fig2}, which indicates that  the condition $\alpha(0)=1$ effectively eliminates the initial singularity of the solutions as shown in the case $\alpha(t)\equiv 0.2$. This phenomenon again indicates the advantage of the multiscale kernel and will be rigorously proved in Theorem \ref{thm:utt}.

\begin{figure}[h]
\setlength{\abovecaptionskip}{0pt}
\centering
\includegraphics[width=2.5in,height=2in]{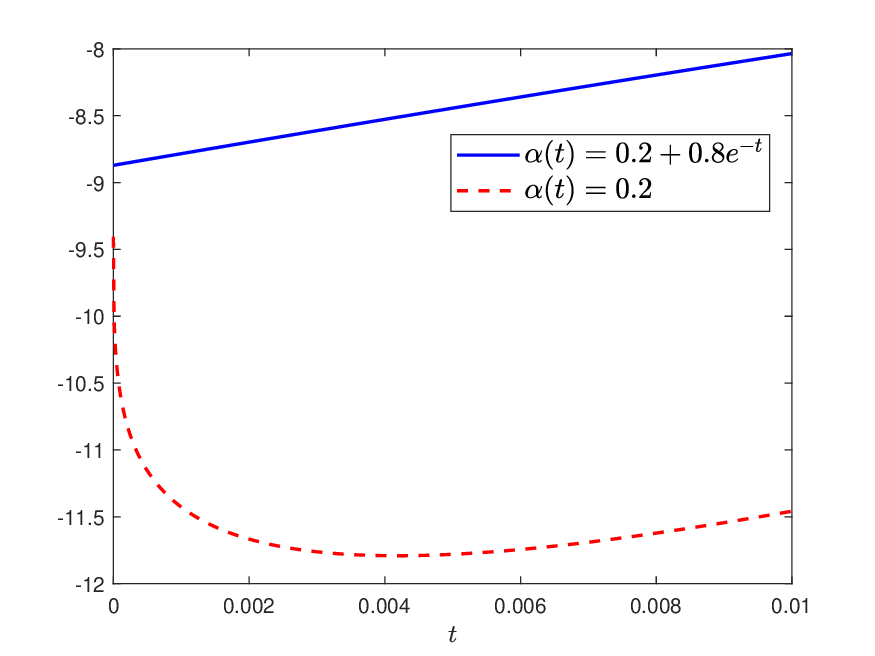}
\caption{Plots of $\p_t u(0.5,t)$ under different $\alpha(t)$.}
\label{Fig2}\vspace{-0.1in}
\end{figure}

\subsection{Novelty and contribution}

For the case $\alpha(t) \equiv \bar \alpha$ for some $0 < \bar \alpha \leq 1$, there exist extensive mathematical analysis results for the PIDE \eqref{eq1.1} \cite{Engler,Friedman,Heard,Heard1,Yin}. For numerical approximation, different numerical methods have been considered for PIDEs and their variants, such as finite element methods \cite{Cannon,Chen,Mustapha1}, discontinuous Galerkin methods \cite{Larsson,Mustapha2}, convolution quadrature methods \cite{Xu1,Xu2}, orthogonal spline collocation methods \cite{Pani,Yan}, finite difference methods \cite{Qiao,Qiu} and so on. Nevertheless, rigorous mathematical and numerical analysis to model \eqref{eq1.1} is not available.
In particular, due to the non-positive and non-monotonic nature of \eqref{k}, many existing time discretization methods for the constant-exponent analogue of \eqref{eq1.4}, such as  convolution quadrature rule \cite{Lub}  and Laplace transform technique \cite{McLean1}, may not be applicable.
In a recent work \cite{ZheLiQiu},  the resolvent estimates are adopted to analyze a mutiscale diffusion model, which, after reformulation,  takes a similar form as \eqref{eq1.1}. Then a first-order-in-time scheme is accordingly developed and analyzed for the mutiscale diffusion model in \cite{ZheLiQiu}, which also applies for the viscoelastic model (\ref{eq1.1})-(\ref{eq1.3}).

The current work considers a second-order-in-time scheme for the model (\ref{eq1.1})-(\ref{eq1.3}), which improves the numerical method in \cite{ZheLiQiu}. To construct a second-order scheme, high-order solution regularity estimates are analyzed, and then we employ the Crank-Nicolson method and the linear interpolation quadrature to discretize temporal operators. To account for the non-positivity and non-monotonicity of \eqref{k}, a framework of exponentially weighted energy argument is developed in numerical analysis, including modifying the difference quotient (cf. (\ref{norms})), numerical scheme, and error equation, selecting a special test function (cf. \eqref{qqqq}), etc. Numerical accuracy  is rigorously proved and then verified by experiments. Furthermore, a crossover dynamics is observed when we model mechanical vibration by (\ref{eq1.1})-(\ref{eq1.3}), which substantiates the multiscale feature of this viscoelastic model.

The rest of the paper is organized as follows. In Section \ref{sec2} we present some  preliminaries to be used subsequently. In Section \ref{sec3}, we derive the well-posedness and high-order regularity estimates of the solutions to (\ref{eq1.1})-(\ref{eq1.3}). In Section \ref{sec5}, we construct a second-order numerical scheme and then prove its stability and error estimate. Numerical experiments are performed in the last section to substantiate the theoretical findings.

\section{Preliminaries}\label{sec2}

\subsection{Notations}
Let $L^p(\om)$ with $1 \le p \le \infty$ be the Banach space of $p$th power Lebesgue integrable functions on $\om$. Given a positive integer $m$,
let  $ W^{m, p}(\Omega)$ be the Sobolev space of $L^p$ functions with the $m$th weakly derivatives in $L^p(\om)$. Let  $H^m(\Omega) := W^{m,2}(\Omega)$ and $H^m_0(\Omega)$ be its subspace with the zero boundary condition reach order $m-1$. For a non-integer $s\geq 0$, $H^s(\Omega)$ is defined through interpolation, see \cite{AdaFou}. Let $\{\lambda_i ,\phi_i\}_{i=1}^\infty$ be eigenpairs of the problem $-\Delta \phi_i = \lambda_i \phi_i$ on $\Omega$ with zero boundary conditions. Then, we introduce the Sobolev space $\check{H}^s(\Omega)$ with $s\geq 0$ by
$$\ds \check{H}^{s}(\Omega) := \bigg \{ v \in L^2(\Omega): \| v \|_{\check{H}^s(\Omega)}^2 : = \sum_{i=1}^{\infty} \lambda_i^{s} (v,\phi_i)^2 < \infty \bigg \},$$
which is a subspace of $H^s(\Omega)$ and satisfies $\check{H}^0(\Omega) = L^2(\Omega)$ and $\check{H}^2(\Omega) = H^2(\Omega) \cap H^1_0(\Omega)$ \cite{Jinbook}.
For a Banach space $\mathcal{X}$, let $W^{m, p}(0,T; \mathcal{X})$ be the space of functions in $W^{m, p}(0,T)$ with respect to $\|\cdot\|_{\mathcal {X}}$. All spaces in this paper are equipped with standard norms \cite{AdaFou,Eva}.

Throughout this paper, we use $Q$ to denote a generic constant
which may be different in different situations. We set $\|\cdot\|:=\|\cdot\|_{L^2(\Omega)}$ and $L^p(\mathcal X)$ for $L^p(0,T;\mathcal X)$ for brevity, and remove the notation $\om$ in the spaces and norms if there is no confusion. For instance, $\|\cdot\|_{L^p(L^2)}$ implies $\|\cdot\|_{L^p(0,T;L^2(\Omega))}$.
Furthermore, we make the {\it \textbf{Assumption A}}:
\begin{itemize}
\item[(i)]  $0 < \alpha_{*} \leq  \alpha(t) \leq 1$, $\alpha(0) = 1$ and $|\alpha^{\prime}(t)| $, $|\alpha^{\prime \prime}(t)| \le Q$ on $[0, T]$.
\item[(ii)] $f \in L^p(L^2)$ for $1< p <\infty$ and $\Delta u_0 \in L^2$.
\end{itemize}


\subsection{Solution representation}

Let $\Gamma_\theta$ be a contour in the complex plane for $\theta\in(\pi/2,\pi)$ and $\delta > 0$, defined as follows
\begin{align*}
\Gamma_\theta := \big \{z\in\C: |{\rm arg}(z)|=\theta, |z|\ge \delta \big \}
\cup \big \{z \in\C: |{\rm arg}(z)|\le \theta, |z|= \delta \big \}.
\end{align*}
For $0<\hat{\mu}\leq 1$ and $Q=Q(\theta,\hat{\mu})$, the following inequalities hold \cite{Akr,Lub}
\begin{equation}\label{GammaEstimate}
\int_{\Gamma_\theta} |z|^{\hat{\mu}-1} |e^{tz}|  \, |d z| \le Q t^{-\hat{\mu}},~~
\Bigg\| \int_{\Gamma_\theta} z^{\hat{\mu}} (z-\mu\Delta)^{-1} e^{tz}
\, d z \Bigg\|_{L^2\rightarrow L^2} \le \f{Q}{t^{\hat{\mu}}}, ~~t\in (0,T].
\end{equation}

For any $q \in L^1_{loc}(0,\infty)$ such that $|q(t)|\leq Qe^{\bar{\sigma} t}$ for large $t$ and for some positive constants $Q$ and $\bar{\sigma}$, we denote the Laplace transform $\mathcal L$ of $q(t)$ as
\begin{equation*}
\begin{split}
 \mathcal{L}q(z):=\int_0^\infty \tilde q(t)e^{-tz}d t, \quad \Re(z)>\bar{\sigma},
\end{split}
\end{equation*}
where $\Re$ denotes the real part of a complex number. Following \cite{Lub}, if $\mathcal{L}q$ is analytic in a sector $\arg(z)<\bar{\chi}$ for some $\bar{\chi}\in (\pi/2,\pi)$ and satisfies that $|(\mathcal{L}q)(z)|\leq Q |z|^{-\hat{\mu}}$ for some $Q,\hat{\mu}>0$, we denote
the corresponding inverse transform $\mathcal L^{-1}$ of $\mathcal{L}q$ in terms of the contour $\Gamma_\theta$ with $\theta<\bar{\chi}$ as
\begin{equation}\label{invLap}
\begin{split}
 \mathcal{L}^{-1}(\mathcal Lq(z)):=\frac{1}{2\pi \rm i}\int_{\Gamma_\theta} e^{tz}\mathcal Lq(z)d z=q(t).
\end{split}
\end{equation}
 Let $E(t):=e^{ \mu t\Delta}$ be the semigroup of operators defined by $\p_t E(t) g -\mu \Delta E(t) g =0$ with $E(t)g =0$ for $\bm x \in \partial \Omega$ and $E(t)g\big|_{t=0} =g$ for $\bm x \in \Omega$. The solution $u$ of the following initial-boundary-value problem
\begin{equation}\label{HeatPDE}\begin{array}{c}
  \partial_t u(\bm x,t) -  \mu \Delta u(\bm x,t)  = f(\bm x,t),~~(\bm x,t)\in\Omega\times(0,T], \\[0.05in]
u(\bm x,t)=0, \quad (\bm x,t)\in \partial\Omega\times(0,T], ~~u(\bm x,0)=0,~~\bm x\in\Omega
\end{array}\end{equation}
can be expressed in terms of the $E(t)$ via the Duhamel's principle
\begin{equation}\label{HeatSoln}
u(\bm x,t) = \int_0^{t} E(t-\theta) f(\bm x,\theta) d\theta,
\end{equation}
where $E(t)$ could be expressed for $\psi \in L^2(\om)$
\begin{equation}\label{Et:Express}
E(t)\psi(\bm x) =\sum_{i=1}^\infty e^{- \mu \lambda_i t}(\psi,\phi_i)\phi_i(\bm x)= \frac{1}{2\pi {\rm i}}\int_{\Gamma_\theta}e^{zt}  (z- \mu\Delta)^{-1}\psi(\bm x) \, d z.
\end{equation}
For $s\geq r\geq -1$ and any $t>0$, the following estimates hold \cite{Akr,Jinbook,Lub}
\begin{equation}\begin{array}{l}\label{E:est}
 \ds \|E(t)\|_{L^2 \to L^2}\le Q,\quad \ds \|E(t) \psi \|_{\check{H}^s}\leq Qt^{-(s-r)/2} \| \psi \|_{\check{H}^r}, \quad \psi \in \check{H}^r.
\end{array}
\end{equation}

Under the {\it \textbf{Assumption A}}, we invoke l’Hospital’s rule to obtain
\begin{equation*}
\begin{array}{l}
\ds \lim_{t\rightarrow 0^+}[\alpha(t)-1]\ln(t)= \lim_{t\rightarrow 0^+} \f{\alpha(t)-1}{({1}/{\ln(t)})} =\lim_{t\rightarrow 0^+}  t \ln^2(t)  \alpha^\prime(t) =0,\\[0.15in]
  \ds \lim_{y\rightarrow s^+}\f{(y-s)^{\alpha( y-s)-1}}{\Gamma(\alpha( y-s))}=\lim_{y\rightarrow s^+} \f{e^{[\alpha(y-s)-1]\ln(y-s)}}{\Gamma(\alpha( y-s))}=1,
  \end{array}
\end{equation*}
and we thus introduce the following estimates that are often used in subsequent analysis \cite{Pod}
\begin{equation}\label{Model:e5}
\begin{split}
& (t-s)^{\alpha( t-s)-1} =  e^{[\alpha( t-s)-1]\ln(t-s)}\le Q,\\
& \int_{0}^{T} e^{-\gamma t} t^{\hat{\mu}-1} d t \leq \gamma^{-\hat{\mu}} \int_{0}^{\infty} e^{-s} s^{\hat{\mu}-1} ds =\gamma^{-\hat{\mu}} \Gamma(\hat{\mu}),\quad 0< \hat{\mu} < 1.
\end{split}
\end{equation}

\section{Mathematical analysis} \label{sec3}
We shall perform mathematical analysis for \eqref{eq1.1}-\eqref{eq1.3}. In general, the well-posedness could be proved following the methods in \cite{ZheLiQiu} such that we only give the results without proof. For higher solution regularity that \cite{ZheLiQiu} does not cover, we give a detailed proof in order to support the construction of high-order numerical scheme.

We first refer the following properties of $k$ from \cite[Lemma 3.2]{ZheLiQiu}.
\begin{lemma}\label{lemma3.1}
  Assume  \textbf{Assumption A} holds. Then, there exists a positive constant $Q$ such that
  \begin{align}\label{eq3.3}
     |k(t)| \leq Q, \quad |k'(t)| \leq Q (1+|\ln(t)|), \quad |k''(t)| \leq Q t^{-1}.
  \end{align}
\end{lemma}

Then the well-posedness could be proved following  \cite[Theorem 3.1]{ZheLiQiu}.
\begin{theorem}\label{thm:Stab}
Under \textbf{Assumption A}, the problem \eqref{eq1.1} has a unique solution $u\in W^{1,p}(L^2)\cap L^p(\check{H}^2) $ for $1 < p < \infty$, and
\begin{equation*}\label{StabEstimate}
\|u\|_{W^{1,p}(L^2)} +\|u\|_{L^p(\check{H}^2)}\le Q \big (\|f\|_{L^p(L^2)} + \|\Delta u_0\|_{L^2} \big ).
\end{equation*}
\end{theorem}

To prove the high-order regularity estimates of the solutions, we move the convolution term in \eqref{eq1.1} to its right-hand side and then use the solution representation (\ref{HeatSoln}) to get
\begin{align}\label{eq4.1}
      u  & = \left[ E(t)u_0 +   \int_0^{t} E(t-\theta) f(\bm x,\theta) d\theta \right] +  \int_0^{t} E(t-\theta) \zeta(k* \Delta u)(\bm x, \theta) d\theta := \Xi_1 \!+ \!\Xi_2.
\end{align}

Then we follow the proofs of \cite[Lemma 4.1 and Theorem 4.2]{ZheLiQiu} to prove Lemma \ref{lemma4.1} and Theorem \ref{thm:utt}, respectively. Due to the similarity, we omit the proofs for simplicity.
\begin{lemma} \label{lemma4.1}
Assume \textbf{Assumption A} holds and $\alpha\in W^{3,\infty}( 0, T)$. For $0<t\leq T$ and $0 \le \varepsilon\ll 1$, there exists a positive constant $Q=Q(\varepsilon,\|\alpha\|_{W^{3,\infty}},T)$ such that
\begin{equation}\label{eq4.3}
\big |\p_tI^{(1-\varepsilon)} \p_t \big(\zeta(k* \Delta v)\big)\big|
\leq Q \int_0^t \frac{\big |\p_s\Delta v(\bm x,s) \big |ds}{(t-s)^{\varepsilon}} +\f{Q |\Delta v_0|}{t^{\varepsilon}}, \quad v=u \text{ or } \Delta u.
\end{equation}
\end{lemma}



\begin{theorem}\label{thm:utt}
Suppose \textbf{Assumption A} holds and $\alpha\in W^{3,\infty}( 0, T)$,  $ u_0 \in \check H^4$ and $f\in W^{1,p}(L^2)\cap L^p(\check H^{2+\sigma})$ for $0 < \sigma \ll 1$ and $ 1 \le p \le\infty$, then the solution of \eqref{eq1.1} satisfies
\begin{equation}\label{thm:utt:e1}
\begin{array}{l}
\ds \|u\|_{W^{2,p}(L^2)}+\|u\|_{L^\infty(\check H^2)}+\|u\|_{W^{1,p}(\check H^2)} \\[0.1in]
\ds \qquad \qquad \leq Q\big( \|f\|_{W^{1,p}(L^2)} + \| f\|_{L^p(\check H^{2+\sigma})} +\| u_0\|_{\check H^4}\big).
\end{array}
\end{equation}
\end{theorem}

Now we give a detailed proof for higher regularity of the solutions that \cite{ZheLiQiu} does not cover.
\begin{theorem}\label{thm:uttt}
Suppose  \textbf{Assumption A} holds and $\alpha\in W^{4,\infty}( 0, T)$,  $ u_0 \in \check H^6$ and $f\in W^{2,p}(L^2)\cap W^{1,p}(\check H^2) \cap L^p(\check H^{4+\sigma})$  for $0 < \sigma \ll 1$ and $ 1 \le p < \infty$.
Then the following estimates hold
\begin{align}
&  \|u\|_{W^{2,p}(\check H^2)}
 \leq Q\Big( \|f\|_{W^{1,p}(\check H^2)} + \| f\|_{L^p(\check H^{4+\sigma})} +\|u_0\|_{\check H^6}\Big), \label{uttt:e1} \\
& \|u\|_{W^{3,p}(L^2)}   \le Q\left(\| u_0\|_{\check H^6}  + \|f\|_{W^{1,p}(\check H^2)} + \| f\|_{L^p(\check H^{4+\sigma})} + \|f\|_{W^{2,p}(L^2)} \right) \nonumber \\
& \qquad \qquad \qquad + Q  t^{-\varepsilon} \|\Delta u_0\|_{L^2}, \quad 0 < \varepsilon \ll 1.
\label{uttt:e2}
\end{align}
\end{theorem}

\begin{proof} We first prove \eqref{uttt:e1}. Note that $\p_t^2 \Delta u= \p_t^2 \Delta\Xi_1 + \p_t^2 \Delta\Xi_2$.
By \eqref{Et:Express}, we have $E^{\prime \prime}(t)\Delta u_0 =   \mu^2\Delta E(t)  \Delta^2 u_0 = \mu^2 E(t) \Delta^3 u_0$ by $u_0 \in \check H^6$, and  we directly compute $ \p_t^2\Delta \Xi_1$ in \eqref{eq4.1} to get
\begin{equation}\label{P2L1}
\begin{split}
\p_t^2\Delta\Xi_1 &  =  \mu^2  E(t) \Delta^3 u_0 + \p_t \Delta f +  \mu \Delta^2 f +  \mu^2
 \int_0^t\sum_{i=1}^\infty \lambda_i^2e^{-\mu\lambda_i (t-s)}(\Delta f,\phi_i)\phi_i ds\\
&  =  \mu^2 E(t) \Delta^3 u_0 + \p_t \Delta f + \mu \Delta^2 f +  \mu \int_0^t E^\prime (t-s) \Delta^2 f(\bm x, s) ds.
\end{split}
\end{equation}
We first use  \eqref{E:est} to obtain
\begin{equation*}\begin{array}{ll}
    \ds \int_0^t \big\| E^\prime (t-s) \Delta^2 f(\cdot, s)\big\|_{L^2(\Omega)} ds & \ds \hspace{-0.1in} \ds \le Q\int_0^t \big\| E (t-s) \Delta^2 f(\cdot, s)\big\|_{\check H^2} ds\\[0.1in]
    \ds & \ds \hspace{-0.1in} \le Q\int_0^t (t-s)^{-\f{2-\sigma}{2}} \|\Delta^2 f(\cdot, s)\|_{\check H^\sigma}ds
    \end{array}
\end{equation*}
for $0 < \sigma \ll 1$, which together with \eqref{P2L1} gives
\begin{equation}\label{wtt:e1}
\begin{array}{ll}
 \|\p_t^2\Delta\Xi_1\|_{L^2(\Omega)} \ds\le Q\big( \|E(t)\|_{L^2 \rightarrow L^2} \|\Delta^3 u_0\|_{L^2} + \|\p_t \Delta f\|_{L^2} + \|\Delta^2 f\|_{L^2}\big)  \\[0.05in]
\ds   \qquad  \qquad \qquad  + Q\int_0^t \big\| E^\prime (t-s) \Delta^2 f(\cdot, s)\big\|_{L^2} ds\\[0.1in]
\ds \ds \le Q(\|\Delta^3 u_0\|_{L^2} + \|\p_t \Delta f\|_{L^2} + \|\Delta^2 f\|_{L^2}) \\[0.05in]
\ds   \qquad  \qquad \qquad  + Q\int_0^t (t-s)^{-\f{2-\sigma}{2}} \|\Delta^2 f(\cdot, s)\|_{\check H^\sigma}ds.
\end{array}
\end{equation}
We take $\|\cdot\|_{L^p(0, T)}$ norm on both sides of \eqref{wtt:e1} and apply Young's inequality to obtain
\begin{equation}\label{wtt:e2}
\begin{array}{rl}
 \|\p_t^2 \Delta \Xi_1\|_{L^p(L^2)}  & \hspace{-0.125in} \ds  \le Q\big( \|\Delta^3 u_0\|_{L^2}  \!+\! \|f\|_{W^{1,p}(\check H^2)}   \!+\! \|\Delta^2 f\|_{L^p(L^2)}\big)\\[0.1in]
 & \hspace{-0.125in} + Q  \big\|t^{-\f{2-\sigma}{2}} * \|\Delta^2 f\|_{\check H^\sigma} \big\|_{L^p(0, T)}\\[0.1in]
\ds & \hspace{-0.125in} \ds \le Q\big(\|\Delta^3 u_0\|_{L^2} + \|f\|_{W^{1,p}(\check H^2)} + \| f\|_{L^p(\check H^{4+\sigma})}\big).
\end{array}
\end{equation}
Next, we shall discuss the analysis of $\p_t^2 \Delta \Xi_2$. We utilize the commutativity of the convolution operator to deduce that
$$\begin{array}{l}
\ds \p_t \int_0^{t} E^\prime(t-s)\,  \big( \zeta(k* \Delta u)(\bm x, s)  \big)ds
= \p_t \int_0^{t} E^\prime(s)  \big ( \zeta(k* \Delta u)(\bm x, y)  \big ) \big |_{y=t-s}\,ds\\[0.1in]
\ds \quad = \int_0^{t} E^\prime(s)  \p_t\big(\zeta(k* \Delta u)(\bm x, y)\big) \big |_{y=t-s}\big) ds
= -\int_0^{t} E^\prime(t-s) \p_s\big(\zeta(k* \Delta u)(\bm x, s) \big) ds.
\end{array}$$
Differentiate $\Xi_2$ in \eqref{eq4.1} twice with respect to $t$ and apply the above resulting equation to obtain
\begin{align}
\ds \p_t \Delta \Xi_2 = \int_0^{t} E^\prime(t-s)\, \big( \zeta(k* \Delta^2 u)(\bm x, s)  \big)ds + \zeta(k* \Delta^2 u)(\bm x,t), \nonumber \\
\ds \p_t^2 \Delta\Xi_2 = -\int_0^{t} E^\prime(t-s) \p_s\big( \zeta(k* \Delta^2 u) (\bm x, s) \big) ds + \p_t\big( \zeta(k* \Delta^2 u)(\bm x,t)\big). \label{wtt:e3}
\end{align}
Now, we utilize the following estimate
 \begin{equation}\begin{array}{l}\label{wtt:E3}
   \ds \big|\p_t( \zeta(k* \Delta^2 u) )\big| \le Q \int_0^t  \big|\p_\theta \Delta^2 u\big| d \theta + Q  |\Delta^2 u_0|
   \end{array}
 \end{equation}
 to directly bound the second term on the right-hand side of \eqref{wtt:e3} and follow \eqref{invLap} and the similar procedures in \cite[(3.14)-(3.15)]{ZheLiQiu} to reformulate
\begin{equation*}
\begin{split}
 &\int_0^{t} E^\prime(t-s)  \p_s\big ( \zeta(k* \Delta^2 u)(\bm x, s)\big) ds \\
 & = \int_0^t \left[\frac{1}{2\pi\rm i} \int_{\Gamma_\theta}z^{1-\varepsilon} (z- \mu\Delta)^{-1} e^{z(t-s)} d z\right]\big(\partial_s I^{(1-\varepsilon)} \p_s(\zeta(k* \Delta^2 u) (\bm x, s)\big)ds.
\end{split}
\end{equation*}
We invoke (\ref{GammaEstimate}) and \eqref{eq4.3} to bound the integral in the square brackets and the first term on the right-hand side in \eqref{wtt:e3}, respectively, to obtain
\begin{equation}\label{wtt:E4}
\begin{split}
& \bigg \| \int_0^{t} E^\prime(t-s)  \p_s\big (\zeta(k* \Delta^2 u) (\bm x, s)\big) ds\bigg\|_{L^2}\\
&\quad \leq Q\int_0^t \f{\big \|\partial_s I^{(1-\varepsilon)} \p_s(\zeta(k* \Delta^2 u)(\bm x,s)) \big \|_{L^2} ds}{(t-s)^{1-\varepsilon}}\\[0.1in]
&\quad\leq Q\int_0^t \f1{(t-s)^{1-\varepsilon}} \Big(\int_0^s \frac{\|\partial_{\theta} \Delta^2 u(\cdot,\theta)\| _{L^2} }{(s-\theta)^{\varepsilon}} d\theta + \|\Delta^2 u_0\| s^{-\varepsilon}\Big) d s\\[0.15in]
&\quad\leq Q\int_0^t  \|\partial_{\theta} \Delta^2 u(\cdot,\theta)\| _{L^2} d\theta + Q \|\Delta^2 u_0\|,
\end{split}
\end{equation}
where we swapped the order of the integral in the last inequality.

Fixing $1 \le p < \infty$, and then we can select $ 0 < \varepsilon \ll 1$ to satisfy $0 < \varepsilon p <1 $. By multiplying $\p_t^2 \Delta\Xi_2$ in \eqref{wtt:e3} by $e^{-\gamma t}$   and taking the $\|\cdot\|_{L^p(0,T)}$ norm on both sides of the resulting equation, we then incorporate \eqref{Model:e5}, \eqref{wtt:E3}, \eqref{wtt:E4}, and Young's convolution inequality to arrive at
\begin{equation}\label{wtt:L2}
\begin{split}
& \big \|e^{-\gamma t} \p_t^2 \Delta\Xi_2 \big \|_{L^p(L^2)}  \leq Q \big \|\big(e^{-\gamma t}\big)*\big(e^{-\gamma t}\|\p_t \Delta^2 u(\cdot,t)\| \big) \big \|_{L^p(0,T)} +Q\|\Delta^2 u_0\| \\
&\qquad \qquad \qquad \qquad \quad   \leq Q\gamma^{-1}\|e^{-\gamma t}\p_t \Delta^2 u\|_{L^p(L^2)}+Q \|\Delta^2 u_0\|.
\end{split}
\end{equation}
We differentiate \eqref{eq4.1} twice in time, apply the $\|\cdot\|_{L^p(0,T)}$ norm on both sides of the resulting equation, then multiply it by $e^{-\gamma t}$, and invoking \eqref{wtt:e2} and \eqref{wtt:L2} to get
\begin{equation}\label{wtt:e5}
\begin{split}
& \big \|e^{-\gamma t} \p_t^2 \Delta u \big \|_{L^p(L^2)}  \ds \leq \big \|e^{-\gamma t}\p_t^2\Delta\Xi_1 \big \|_{L^p(L^2)}
+\big \|e^{-\gamma t}\p_t^2\Delta\Xi_2 \big \|_{L^p(L^2)} \\
& \qquad\qquad\qquad\qquad \leq  Q\big( \|f\|_{W^{1,p}(\check H^2)} \!+\! \| f\|_{L^p(\check H^{4+\sigma})} \!+\!\|\Delta^2 u_0\|\!+\!\|\Delta^3 u_0\|\big)  \\
& \qquad\qquad\qquad\qquad  +  Q\gamma^{-1} \big \| e^{-\gamma t}\p_t \Delta^2 u \big \|_{L^p(L^2)}.
\end{split}
\end{equation}
We then differentiate \eqref{eq4.1} with respect to $t$,  apply the $\|\cdot\|_{L^p(0,T)}$ norm on both sides of the resulting equation multiplied by $e^{-\gamma t}$, and combine  the estimates \eqref{wtt:E3}, \eqref{wtt:e5} with Young's convolution inequality to bound
\begin{equation*}
\begin{split}
\big \|e^{-\gamma t} \p_t \Delta^2 u \big \|_{L^p(L^2)} &  \leq Q \big \| e^{-\gamma t} \big(\p_t^2 \Delta u- \p_t \Delta f -\p_t(\zeta(k* \Delta^2 u)) \big)\big \|_{L^p(L^2)}\\
&  \leq  Q\big( \|f\|_{W^{1,p}(\check H^2)} + \| f\|_{L^p(\check H^{4+\sigma})} +\| u_0\|_{\check H^6}\big)\\[0.1in]
&  + Q\gamma^{-1} \big \| e^{-\gamma t}\p_t \Delta^2 u \big \|_{L^p(L^2)}.
\end{split}
\end{equation*}
Here we set $\gamma$ large enough to eliminate the last term on the right-hand side of the above equation to obtain
\begin{equation*}
\big \|e^{-\gamma t} \p_t \Delta^2 u \big \|_{L^p(L^2)} \leq   Q\big( \|f\|_{W^{1,p}(\check H^2)} + \| f\|_{L^p(\check H^{4+\sigma})} +\| u_0\|_{\check H^6}\big),
\end{equation*}
which combined with \eqref{wtt:e5} leads to
\begin{equation}\label{wtt:e7}\begin{array}{rl}
\big \|\p_t^2 u \big \|_{L^p(\check H^2)}   \leq Q\left(\| u_0\|_{\check H^6}  + \|f\|_{W^{1,p}(\check H^2)} + \| f\|_{L^p(\check H^{4+\sigma})} \right).
\end{array}
\end{equation}
This yields \eqref{uttt:e1}.

To prove \eqref{uttt:e2}, we shall bound $\|\p_t^3 u\|_{L^p(L^2)}$. We apply the initial condition $\p_tu(\bm x, 0) = \mu \Delta u_0 + f(\bm x,0)$ from \eqref{eq1.1}, the following relation
\begin{align*}
    \p_{t}^2 \int_0^t k(s)\Delta u(\bm x,t-s) ds &= k'(t)\Delta u_0 + k(t) \Delta \p_tu(\bm x,0) + \int_0^t k(s) \p_t^2 \Delta u(\bm x,t-s)ds,
\end{align*}
and \eqref{eq1.1} to  obtain
\begin{align*}
    \p_{t}^3 u & =   \mu \p_{t}^2\Delta u +  \zeta \p_{t}^2 \int_0^t k(s)\Delta u(\bm x, t-s) ds + \p_{t}^2f(t) \\
    & =   \mu \p_{t}^2\Delta u + \p_{t}^2f(t) +  \zeta k'(t)\Delta u_0 +  \zeta k(t) [  \mu \Delta^2 u_0 + \Delta f(\bm x, 0)] \\
    & + \zeta \int_0^t k(s) \p_t^2 \Delta u(\bm x, t-s)ds.
\end{align*}
We combine  \eqref{uttt:e1} and \eqref{eq3.3} in Lemma \ref{lemma3.1}, i.e., $|k'(t)|\leq Q(1+|\ln(t)|)\leq Qt^{-\varepsilon}$, and the Sobolev embedding $W^{1, p}(0, T)  \hookrightarrow L^\infty(0, T)$ \cite{AdaFou} to obtain
\begin{align*}
    \|\p_{t}^3 u\|_{L^p(L^2)} &\leq   Q\big(\|\p_{t}^2\Delta u\|_{L^p(L^2)} + \|\p_{t}^2f(t)\|_{L^p(L^2)} + |k'(t)|\|\Delta u_0\|_{L^2}) \\
    & \quad+ Q|k(t)| \left[\|\Delta^2 u_0\|_{L^2} + \|\Delta f(\cdot,0)\|_{L^2} \right]  \\
    & \quad + Q\int_0^t |k(s)| \|\p_t^2 \Delta u(\cdot, t-s)\|_{L^p(L^2)}ds \\
    &\leq Q\big(\| u_0\|_{\check H^6}   \!+  \!\|f\|_{W^{1,p}(\check H^2)} \!+  \!\| f\|_{L^p(\check H^{4+\sigma})} + \|f\|_{W^{2,p}(L^2)} \!+ \! \|\Delta f(\cdot, 0)\|_{L^2} \big) \\
    & \qquad \qquad + Q  t^{-\varepsilon} \|\Delta u_0\|_{L^2}\\
    &  \le Q\big(\| u_0\|_{\check H^6}  + \|f\|_{W^{1,p}(\check H^2)} + \| f\|_{L^p(\check H^{4+\sigma})} + \|f\|_{W^{2,p}(L^2)} \big)\\
     & \qquad \qquad + Q  t^{-\varepsilon} \|\Delta u_0\|_{L^2},
\end{align*}
which completes the proof.
\end{proof}

\section{Second-order scheme}\label{sec5}

In this section, we propose and analyze a second-order numerical approximation to \eqref{eq1.1}-\eqref{eq1.3}.

\subsection{Temporal semi-discrete scheme}
Denote
\begin{equation*}
   0=t_0 < t_1 < \cdots < t_{N} = T, \quad N \in \mathbb{Z}^+, \quad  \tau = t_{n}-t_{n-1} = \frac{T}{N},  \quad 1\leq n \leq N.
\end{equation*}
In addition, we define
\begin{equation*}
   \delta_t V^n = \frac{V^n-V^{n-1}}{\tau}, \;  V^{n-1/2} = \frac{V^n+V^{n-1}}{2}, \;  t_{n-1/2} = \frac{t_n+t_{n-1}}{2}, \;  1\leq n \leq N.
\end{equation*}
To facilitate the analysis, we first denote
\begin{equation}\label{qwl01}
   \hat{V}^{n} =  e^{-\lambda t_n}  V^n, \quad 1\leq \lambda < \infty,
\end{equation}
and
\begin{align}\label{norms}
    \delta_t^{\lambda} \hat{V}^n : = \frac{\hat{V}^n- e^{-\lambda \tau}\hat{V}^{n-1}}{\tau} = e^{-\lambda t_n}\delta_t V^n, \quad n\geq 1.
\end{align}
We employ the linear interpolation quadrature to approximate the integral term in \eqref{eq1.1}
\begin{equation}\label{eq6.1}
  \begin{split}
      I^{(\alpha(t_n))} \varphi(t_n) & = \int_{0}^{t_n} \frac{(t_n-s)^{\alpha(t_n-s)-1}}{\Gamma(\alpha(t_n-s))} \varphi(s) ds\\ & \approx \sum_{j=1}^{n} \int_{t_{j-1}}^{t_j} \frac{(t_n-s)^{\alpha(t_n-s)-1}}{\Gamma(\alpha(t_n-s))} \mathcal{L}_1[\varphi](s) ds
      \\
      & = \sum_{j=1}^{n} \left[ a_{n,j}\varphi^j + b_{n,j}\varphi^{j-1} \right] := {II}_{n}(\varphi), \quad n\geq 1,
  \end{split}
\end{equation}
where
\begin{equation}\label{L1I}
  \begin{split}
      \varphi^j = \varphi(t_j), \quad \mathcal{L}_1[\varphi](s) =  \frac{s-t_{j-1}}{\tau}\varphi^j + \frac{t_j-s}{\tau}\varphi^{j-1},
  \end{split}
\end{equation}
and
\begin{equation}\label{eq6.2}
  \begin{split}
      & a_{n,j} = \int_{t_{j-1}}^{t_j}   \frac{(t_n-s)^{\alpha(t_n-s)-1}}{\Gamma(\alpha(t_n-s))} \frac{s-t_{j-1}}{\tau} ds > 0, \\
       & b_{n,j} = \int_{t_{j-1}}^{t_j}   \frac{(t_n-s)^{\alpha(t_n-s)-1}}{\Gamma(\alpha(t_n-s))} \frac{t_{j}-s}{\tau} ds > 0.
  \end{split}
\end{equation}
Then \eqref{eq1.1} implies
\begin{equation}\label{eq6.3}
  \begin{split}
      & \frac{1}{2}[I^{(\alpha(t_n))} \varphi(t_n)+I^{(\alpha(t_{n-1}))} \varphi(t_{n-1})] \approx \frac{1}{2}\sum_{j=1}^{n} \left[ a_{n,j}\varphi^j + b_{n,j}\varphi^{j-1} \right] \\
      & \qquad\qquad\qquad\quad + \frac{1}{2}\sum_{j=1}^{n-1} \left[ a_{n-1,j}\varphi^j + b_{n-1,j}\varphi^{j-1} \right],  \quad n\geq 1,
  \end{split}
\end{equation}
and we use the fact that $a_{n-1,j}=a_{n,j+1}$ and $b_{n-1,j}=b_{n,j+1}$ to get
\begin{align}
   & \frac{1}{2}[I^{(\alpha(t_n))} \varphi(t_n)+I^{(\alpha(t_{n-1}))} \varphi(t_{n-1})] \approx \sum_{j=1}^{n} a_{n,j} \frac{\varphi^j+\varphi^{j-1}}{2} \nonumber \\
      &  + \sum_{j=1}^{n} b_{n,j} \frac{\varphi^{j-1}+\varphi^{j-2}}{2} - a_{n,1}\frac{\varphi^0}{2} := II_{n-1/2}(\varphi), \quad \varphi^{-1}=0=II_0(\varphi). \label{eq6.4}
\end{align}
Now, we consider \eqref{eq1.1} at the point $t=t_n$ and apply Crank-Nicolson method to get for $u^n:=u(\bm x,t_n)$
  \begin{align}
     & \delta_t u^n -  \mu\Delta u^{n-1/2} - \zeta II_{n-1/2}(\Delta u) = f^{n-1/2} + \mathcal{R}^{n}, \quad 1\leq n \leq N, \label{eq6.5} \\
     & u^0 = u_0, \label{eq6.6}
  \end{align}
where $\mathcal{R}^{n}=\mathcal{R}_1^{n}+\mathcal{R}_2^{n-1/2}$, and
  \begin{align}
     & \mathcal{R}_1^{n} = \left[\delta_t u^n - \p_t u(t_{n-1/2}) \right] + \left[\p_t u(t_{n-1/2}) - (\p_t u)^{n-1/2} \right], \quad n\geq 1,  \label{qqq01}  \\
     & \mathcal{R}_2^{n} = \zeta( II_{n}(\Delta u)- I^{(\alpha(t_n))} \Delta u(t_n)), \quad n\geq 1, \quad  \mathcal{R}_2^{0} = 0. \label{qqq02}
  \end{align}
By omitting the small term $\mathcal{R}^n$ in \eqref{eq6.5} and replacing $U^n$ as approximate solutions of $u^n$, we obtain Crank-Nicolson scheme as
  \begin{align}
     & \delta_t U^n -  \mu\Delta U^{n-1/2} - \zeta II_{n-1/2}(\Delta U) = f^{n-1/2}, \quad 1\leq n \leq N, \label{eq6.8} \\
     & U^0 = u_0. \label{eq6.9}
  \end{align}
  Then, we use the notation of \eqref{qwl01} and \eqref{norms} and multiply both sides of \eqref{eq6.8} by $e^{-\lambda t_n}$ to get
\begin{align}
   \delta_t^{\lambda} \hat{U}^n  & - \mu\Lambda_{\lambda}(\Delta \hat{U}^{n}) -\zeta \sum\limits_{j=1}^{n} \tilde{a}_{n,j}  \Lambda_{\lambda}(\Delta \hat{U}^j) - \zeta\sum\limits_{j=2}^{n} \tilde{b}_{n,j}  \Lambda_{\lambda}(\Delta \hat{U}^{j-1}) \nonumber \\
   & + \zeta(a_{n,1}-b_{n,1})e^{-\lambda t_{n}} \Lambda_{\lambda}(\Delta \hat{U}^0)  = \Lambda_{\lambda}(\hat{f}^n),  \quad 1\leq n \leq N, \label{qwlB}
\end{align}
where
\begin{equation}\label{qqqq}
   \tilde{a}_{n,j} = a_{n,j} e^{-\lambda t_{n-j}}, \quad \tilde{b}_{n,j} = b_{n,j} e^{-\lambda t_{n+1-j}},  \quad \Lambda_{\lambda}(\hat{V}^n) = \frac{\hat{V}^n+e^{-\lambda \tau}\hat{V}^{n-1}}{2}.
\end{equation}

\subsection{Stability and error estimate}
Next, we shall deduce the following stability result.
\begin{theorem}\label{thm7}
   Let $U^m$ be the solution of the Crank-Nicolson scheme \eqref{eq6.8}-\eqref{eq6.9}. Then we have
   \begin{equation*}
      \|U^m\|^2  \leq  Q \left( \|U^0\|^2 + \tau\|\nabla U^0\|^2  + \tau \sum_{n=1}^{m}\|f^{n-1/2}\|^2 \right), \quad 1\leq m \leq N.
   \end{equation*}
\end{theorem}
\begin{proof}
   We first take the inner product of \eqref{qwlB} with $\tau \Lambda_{\lambda}(\hat{U}^n)$ and then sum for $n$ from 1 to $m$ to get
   \begin{equation*}
     \begin{split}
         & \tau\sum_{n=1}^{m}   (  \delta_t^{\lambda}\hat{U}^n, \Lambda_{\lambda}(\hat{U}^n) ) + \mu  \tau\sum_{n=1}^{m}( \Lambda_{\lambda}(\nabla \hat{U}^{n}), \Lambda_{\lambda}(\nabla \hat{U}^{n}))   \\
         &  + \zeta\tau\sum_{n=1}^{m}  \tilde{a}_{n,n} ( \Lambda_{\lambda}(\nabla \hat{U}^n), \Lambda_{\lambda}(\nabla \hat{U}^n) ) = - \zeta\tau\sum_{n=2}^{m} \sum\limits_{j=1}^{n-1}\tilde{a}_{n,j} ( \Lambda_{\lambda}(\nabla \hat{U}^j), \Lambda_{\lambda}(\nabla \hat{U}^n) ) \\
         & - \zeta\tau\sum_{n=2}^{m} \sum\limits_{j=2}^{n}\tilde{b}_{n,j} ( \Lambda_{\lambda}(\nabla \hat{U}^{j-1}), \Lambda_{\lambda}(\nabla \hat{U}^n) )  + \tau\sum_{n=1}^{m} (\Lambda_{\lambda}(\hat{f}^n), \Lambda_{\lambda}( \hat{U}^n)) \\
         & + \zeta\tau\sum_{n=1}^{m} (a_{n,1}-b_{n,1})e^{-\lambda t_{n}} (\Lambda_{\lambda}(\nabla \hat{U}^0),\Lambda_{\lambda}(\nabla \hat{U}^n)): = -\zeta\Phi_1 - \zeta\Phi_2 + \Phi_4 + \zeta\Phi_3,
     \end{split}
   \end{equation*}
which follows from $a_{n,n}\geq 0$ and
   \begin{equation}\label{eq6.12}
     \begin{split}
           \tau\sum_{n=1}^{m} &  (  \delta_t^{\lambda}\hat{U}^n, \Lambda_{\lambda}(\hat{U}^n) ) = \sum_{n=1}^{m}\frac{\|\hat{U}^n\|^2-\|e^{-\lambda \tau}\hat{U}^{n-1}\|^2}{2} \\
           & \qquad \geq \sum_{n=1}^{m}\frac{\|\hat{U}^n\|^2-\|\hat{U}^{n-1}\|^2}{2} = \frac{\|\hat{U}^m\|^2-\|\hat{U}^{0}\|^2}{2}
     \end{split}
   \end{equation}
that
\begin{align}
    \frac{\|\hat{U}^m\|^2-\|\hat{U}^{0}\|^2}{2}  +  \mu\tau\sum_{n=1}^{m} \| \Lambda_{\lambda}(\nabla \hat{U}^{n})\|^2 \leq  \zeta \sum_{j=1}^{3} |\Phi_j| + |\Phi_4|. \label{eq6.13}
\end{align}
Below we will further estimate the terms of the right-hand side of \eqref{eq6.13}. First, we apply Cauchy-Schwarz inequality and $ab\leq \frac{a^2+b^2}{2}$ to obtain
\begin{align*}
   |\Phi_1| & \leq \tau\sum_{n=2}^{m} \sum\limits_{j=1}^{n-1}\tilde{a}_{n,j} \| \Lambda_{\lambda}(\nabla \hat{U}^j) \|  \| \Lambda_{\lambda}(\nabla \hat{U}^n) \| \nonumber \\
   & \leq \tau\sum_{n=2}^{m} \sum\limits_{j=1}^{n-1}\tilde{a}_{n,j} \frac{\| \Lambda_{\lambda}(\nabla \hat{U}^j) \|^2 + \| \Lambda_{\lambda}(\nabla \hat{U}^n) \|^2}{2} \nonumber \\
   & = \tau\sum_{n=2}^{m} \sum\limits_{j=1}^{n-1}\tilde{a}_{n,j}
  \frac{\| \Lambda_{\lambda}(\nabla \hat{U}^j) \|^2 }{2} + \tau\sum_{n=2}^{m}  \frac{\| \Lambda_{\lambda}(\nabla \hat{U}^n) \|^2 }{2} \sum\limits_{j=1}^{n-1}\tilde{a}_{n,j}
  \nonumber \\
  & = \tau\sum_{j=1}^{m-1} \frac{\| \Lambda_{\lambda}(\nabla \hat{U}^j) \|^2 }{2}\sum\limits_{n=j+1}^{m}\tilde{a}_{n,j}
   + \tau\sum_{n=2}^{m}  \frac{\| \Lambda_{\lambda}(\nabla \hat{U}^n) \|^2 }{2} \sum\limits_{j=1}^{n-1}\tilde{a}_{n,j},
\end{align*}
where
\begin{align*}
   & \sum\limits_{n=j+1}^{m}\tilde{a}_{n,j} \leq Q \sum_{n=j+1}^{m} e^{-\lambda t_{n-j}} \tau (t_n-t_j)^{\alpha_* -1} = Q \tau\sum_{p=1}^{m-j} e^{-\lambda t_{p}}  (t_p)^{\alpha_* -1} \\
   & \qquad\qquad\;\; \leq Q \int_{0}^{t_{m-j}} e^{-\lambda s} s^{\alpha_* -1} ds \leq Q\int_{0}^{\infty} e^{-\lambda s} s^{\alpha_* -1} ds \leq Q \lambda^{-\alpha_*}, \\
   & \sum\limits_{j=1}^{n-1}\tilde{a}_{n,j} \leq  Q \sum_{p=1}^{n-1} e^{-\lambda t_{p}} \tau (t_p)^{\alpha_* -1} \leq Q\int_{0}^{\infty} e^{-\lambda s} s^{\alpha_* -1} ds \leq Q \lambda^{-\alpha_*}.
\end{align*}
Thus, we get
\begin{align}
     |\Phi_1| & \leq (Q \lambda^{-\alpha_*}) \tau\sum_{n=1}^{m}  \| \Lambda_{\lambda}(\nabla \hat{U}^n) \|^2. \label{eq6.14}
\end{align}
Then, similar to the analysis of \eqref{eq6.14}, we use $\tilde{b}_{n,n}\leq Q \int_{t_{n-1}}^{t_n}(t_n-s)^{\alpha_* -1}ds\leq Q\tau^{\alpha_*}$ to get
\begin{align}
     |\Phi_2| & \leq Q \tau\sum_{n=2}^{m} \tilde{b}_{n,n} \| \Lambda_{\lambda}(\nabla \hat{U}^{n-1}) \| \| \Lambda_{\lambda}(\nabla \hat{U}^n) \| \nonumber \\
     & + \tau\sum_{n=2}^{m} \sum_{j=2}^{n-1} \tilde{b}_{n,j} \| \Lambda_{\lambda}(\nabla \hat{U}^{j-1}) \| \| \Lambda_{\lambda}(\nabla \hat{U}^n) \| \nonumber \\
     & \leq Q\left( \tau^{\alpha_*}+\lambda^{-\alpha_*} \right) \tau\sum_{n=1}^{m}  \| \Lambda_{\lambda}(\nabla \hat{U}^n) \|^2. \label{eq6.15}
\end{align}
Next, we discuss the term $|\Phi_3|$. We utilize
\begin{align*}
     & |a_{1,1}-b_{1,1}| \leq Q  \tau^{\alpha_*}, \quad  |a_{n,1}-b_{n,1}| \leq Q  \tau t_{n-1}^{\alpha_* -1} \leq Q  \tau^{\alpha_*}, \quad n \geq 2,
\end{align*}
and Cauchy-Schwarz inequality to get
 \begin{align}
     |\Phi_3| & \leq (Q \tau^{\alpha_*}) \tau\sum_{n=1}^{m}  \| \Lambda_{\lambda}(\nabla \hat{U}^n) \|^2 + Q \tau^{\alpha_* +1} \| \Lambda_{\lambda}(\nabla \hat{U}^0) \|^2 \nonumber \\
     & + \tau\sum_{n=2}^{m}Q  \tau t_{n-1}^{\alpha_* -1} e^{-\lambda t_{n-1}} \| \Lambda_{\lambda}(\nabla \hat{U}^0) \|^2 \nonumber \\
      & \leq (Q \tau^{\alpha_*}) \tau\sum_{n=1}^{m}  \| \Lambda_{\lambda}(\nabla \hat{U}^n) \|^2 + Q \tau \| \Lambda_{\lambda}(\nabla \hat{U}^0) \|^2.  \label{eq6.16}
\end{align}
Finally, we bound $|\Phi_4|$. It is easy to obtain with $\Lambda_{\lambda}(\hat{V}^n)=e^{-\lambda t_n}V^{n-1/2}$,
 \begin{align}
     |\Phi_4| & \leq Q \tau\sum_{n=1}^{m}  \| \Lambda_{\lambda}( \hat{U}^n) \|   \left \| \Lambda_{\lambda}(\hat{f}^n) \right\| \leq Q \tau\sum_{n=1}^{m}  \| U^{n-1/2} \|   \left \|f^{n-1/2} \right\|.    \label{eq6.17}
\end{align}
We put \eqref{eq6.14}--\eqref{eq6.17} into the right-hand side of \eqref{eq6.13}:
\begin{align*}
        \frac{\|\hat{U}^m\|^2}{2} & +\mu   \tau\sum_{n=1}^{m} \| \Lambda_{\lambda}(\nabla \hat{U}^{n})\|^2  \leq  Q\zeta\left( \tau^{\alpha_*}+\lambda^{-\alpha_*} \right) \tau\sum_{n=1}^{m}  \| \Lambda_{\lambda}(\nabla \hat{U}^n) \|^2 \nonumber \\
        & + Q \zeta\tau \| \Lambda_{\lambda}(\nabla \hat{U}^0) \|^2 +  Q \tau\sum_{n=1}^{m}  \| U^{n-1/2} \|   \left \|f^{n-1/2} \right\| + \frac{\|\hat{U}^0\|^2}{2},
\end{align*}
in which, taking suitable large $\lambda$ and small $\tau$ so that $Q\zeta\left( \tau^{\alpha_*}+\lambda^{-\alpha_*} \right) \leq \mu$, we get
\begin{align}
        \frac{\|\hat{U}^m\|^2}{2} & \leq   Q \tau \| \Lambda_{\lambda}(\nabla \hat{U}^0) \|^2  + \frac{\|\hat{U}^0\|^2}{2}+  Q \tau\sum_{n=1}^{m}  \| U^{n-1/2} \|   \left \|f^{n-1/2} \right\|, \label{eq6.18}
\end{align}
which combines the estimate
\begin{align*}
   Q \tau\sum_{n=1}^{m}  \| U^{n-1/2} \|   \left \|f^{n-\frac{1}{2}} \right\| \leq Q\tau\|U^m\|^2  + Q \tau\sum_{n=0}^{m-1} \|U^n\|^2 +  Q \tau\sum_{n=1}^{m}  \left \|f^{n-\frac{1}{2}} \right\| ^2,
\end{align*}
$\frac{\|\hat{U}^m\|^2}{2} = \frac{e^{-2\lambda t_m}\|U^m\|^2}{2}$, and $2Q e^{2\lambda t_m} \tau \leq \frac{1}{4}$ yields
\begin{align*}
        \frac{\|U^m\|^2}{4} & \leq   Q \tau \|\nabla U^0 \|^2  + \frac{\|U^0\|^2}{2} +  Q \tau\sum_{n=1}^{m}  \left \|f^{n-\frac{1}{2}} \right\| ^2 +  Q \tau\sum_{n=0}^{m-1} \|U^n\|^2 .
\end{align*}
This finishes the proof by the discrete Gr\"{o}nwall's inequality.
\end{proof}

\vskip 1mm
In what follows, we shall deduce the convergence of the Crank-Nicolson scheme \eqref{eq6.8}-\eqref{eq6.9}. By subtracting \eqref{eq6.8}-\eqref{eq6.9} from \eqref{eq6.5}-\eqref{eq6.6}, we arrive at the following error equations
  \begin{align}
     & \delta_t \rho^n -  \mu\Delta \rho^{n-1/2} - \zeta II_{n-1/2}(\Delta \rho) = \mathcal{R}^{n}, \quad 1\leq n \leq N, \label{eq6.19} \\
     & \rho^0 = 0. \label{eq6.20}
  \end{align}
Then, we use \eqref{qwlB} to get for $1\leq n \leq N$
\begin{align}
   \delta_t^{\lambda} \hat{\rho}^n  & - \mu\Lambda_{\lambda}(\Delta \hat{\rho}^{n}) - \zeta \sum\limits_{j=1}^{n} \tilde{a}_{n,j}  \Lambda_{\lambda}(\Delta \hat{\rho}^j) - \zeta \sum\limits_{j=2}^{n} \tilde{b}_{n,j}  \Lambda_{\lambda}(\Delta \hat{\rho}^{j-1})   = \hat{\mathcal{R}}^n. \label{eq6.21}
\end{align}

\begin{theorem}\label{thm8} Under the conditions of Theorem \ref{thm:uttt}, the following error estimate holds with $T<\infty$
   \begin{equation*}
      \|u^m-U^m\| \leq Q\tau^2, \quad 1\leq m \leq N.
   \end{equation*}
\end{theorem}
\begin{proof}
    Based on the analysis of Theorem \ref{thm7} and \eqref{eq6.18}, we use \eqref{eq6.21} to get
\begin{align*}
        \frac{\|\hat{\rho}^m\|^2}{2} & \leq   Q \tau\sum_{n=1}^{m}  \| \rho^{n-1/2} \|   \left \|\hat{\mathcal{R}}^{n} \right\|\leq Q \tau\sum_{n=1}^{m}  \| \rho^{n-1/2} \|   \left \|\mathcal{R}^{n} \right\|,
\end{align*}
which further gives
\begin{align*}
       \| \rho^m \|^2 & \leq   Q \tau\sum_{n=1}^{m}  \| \rho^{n-1/2} \|  \left \|\mathcal{R}^{n} \right\|.
\end{align*}
Then, selecting a suitable $m_*$ such that $\|\rho^{m_*}\|=\max\limits_{0\leq n \leq m}\|\rho^{n}\|$, we have
\begin{align}
       \| \rho^m \| \leq \|\rho^{m_*}\|  \leq   Q \tau\sum_{n=1}^{m_*}  \left \|\mathcal{R}^{n} \right\| \leq Q \tau\sum_{n=1}^{m}  \left( \left\|\mathcal{R}_1^{n} \right\| + \left\|\mathcal{R}_2^{n-1/2} \right\| \right). \label{eq6.22}
\end{align}
Next, we will discuss the terms of the right-hand side of \eqref{eq6.22}. By Taylor expansion with integral remainder, \eqref{qqq01} gives
\begin{align*}
    & \left\|\mathcal{R}_1^{1} \right\| \leq 2\int_0^{\tau} \|\partial_{t}^2 u(t) \| dt, \quad  \left\|\mathcal{R}_1^{n} \right\| \leq \tau \int_{t_{n-1}}^{t_n} \|\partial_{t}^3 u(t) \| dt, \quad n\geq 2,
\end{align*}
which combines \eqref{thm:utt:e1} and \eqref{uttt:e2}, then we get
\begin{align}
       \tau\sum_{n=1}^{m} \left\|\mathcal{R}_1^{n} \right\|
       \leq Q\tau^2 + Q\tau^2 \int_{t_1}^{t_m} t^{-\varepsilon} dt \leq Q\tau^2.
       \label{eq6.23}
\end{align}
We apply the remainder of linear interpolation with $\xi_j\in (t_{j-1},t_j)$, \eqref{qqq02}, \eqref{L1I}, and \eqref{uttt:e1} to get
\begin{align*}
    \|\mathcal{R}_2^{n}\| & \leq  \zeta \sum_{j=1}^{n} \int_{t_{j-1}}^{t_j} \frac{(t_n-s)^{\alpha(t_n-s) -1} }{\Gamma(\alpha(t_n-s))} \big\|\Delta u(s) - \mathcal{L}_1[\Delta u](s) \big\| ds  \\
    & \leq Q \sum_{j=1}^{n} \int_{t_{j-1}}^{t_j} (t_n-s)^{\alpha_* -1} \left\| \frac{1}{2} \partial_t^2 \Delta u(\xi_j) (s-t_j)(s-t_{j-1})  \right\|ds  \\
    & \leq  Q \sum_{j=1}^{n} \int_{t_{j-1}}^{t_j} (t_n-s)^{\alpha_* -1} \big\|\partial_t^2 \Delta u(\xi_j)\big\| \frac{\tau^2}{2} ds \\
    & \leq Q \tau^2  \int_{0}^{t_n} (t_n-s)^{\alpha_* -1} ds.
\end{align*}
Thus we further obtain
\begin{align}
       \tau\sum_{n=1}^{m} \left\|\mathcal{R}_2^{n} \right\|
       \leq  Q\tau^3 \sum_{n=1}^{m} \int_{0}^{t_n} (t_n-s)^{\alpha_* -1} ds \leq Q\tau^2.
       \label{eq6.24}
\end{align}
By inserting \eqref{eq6.23} and \eqref{eq6.24} into \eqref{eq6.22}, the proof is completed.
\end{proof}

\subsection{Fully discrete scheme}

Next, we shall develop a fully discrete Crank-Nicolson scheme by applying the spatial finite element method. Define a quasi-uniform partition of $\Omega$ with the mesh diameter $h$. Let $S_h$ be the space of continuous piecewise linear functions on $\Omega$ with respect to the partition. Define the Ritz projector $I_h: H^1_0(\Omega)\rightarrow S_h$ by
\begin{align*}
     (\nabla (I_h w - w), \nabla \chi ) = 0 \quad \text{for} \;\; \chi \in S_h
\end{align*}
with the following approximation property  \cite{Jinbook}
\begin{align}
   \|w-I_h w\|_{L^2}\leq Qh^2 \|w(t)\|_{H^2(\Omega)}.  \label{lkx02}
\end{align}

Based on the Crank-Nicolson scheme \eqref{eq6.8}-\eqref{eq6.9}, we obtain the fully discrete Galerkin scheme:  find $\hat{U}_h\in S_h$  such that
\begin{align}
   \left(\delta_t^{\lambda} \hat{U}_h^n, \chi \right)  & + \mu \left(\Lambda_{\lambda}(\nabla \hat{U}_h^{n}), \nabla\chi \right) \nonumber \\
   & + \zeta \sum\limits_{j=1}^{n} \tilde{a}_{n,j}  \left(\Lambda_{\lambda}(\nabla \hat{U}_h^j), \nabla\chi \right) + \zeta \sum\limits_{j=2}^{n} \tilde{b}_{n,j}  \left(\Lambda_{\lambda}(\nabla \hat{U}_h^{j-1}), \nabla\chi \right) \label{eq6.25} \\
   & -\zeta  (a_{n,1}-b_{n,1})e^{-\lambda t_{n}} \left(\Lambda_{\lambda}(\nabla \hat{U}_h^0), \nabla\chi \right)  = \left(\Lambda_{\lambda}(\hat{f}^n), \chi \right), \;\;  1\leq n \leq N,  \nonumber
\end{align}
for all $\chi\in S_h$ with the suitable approximation $U_h^0\approx u_0$.

To facilitate analysis, we decompose
\begin{align*}
    u(t_n)- U_h^n = \zeta^n - \eta(t_n):=\zeta^n - \eta^n,
\end{align*}
where
\begin{align}
    \zeta^n = I_h u(t_n) - U_h^n \in S_h, \quad \eta^n = I_h u(t_n) - u(t_n). \label{bbb}
\end{align}
We apply \eqref{eq6.5}, \eqref{bbb}, and \eqref{eq6.25} to get the following error equation
\begin{align}
   \left(\delta_t^{\lambda} \hat{\zeta}^n, \chi \right)  & + \mu\left(\Lambda_{\lambda}(\nabla \hat{\zeta}^{n}), \nabla\chi \right) + \zeta \sum\limits_{j=1}^{n} \tilde{a}_{n,j}  \left(\Lambda_{\lambda}(\nabla \hat{\zeta}^j), \nabla\chi \right)
 \label{eq6.26} \\
   &  +\zeta  \sum\limits_{j=2}^{n} \tilde{b}_{n,j}  \left(\Lambda_{\lambda}(\nabla \hat{\zeta}^{j-1}), \nabla\chi \right)   = \left( \hat{\mathcal{R}}^n+\delta_t^{\lambda} \hat{\eta}^n, \chi \right), \;\;  1\leq n \leq N.  \nonumber
\end{align}

\begin{theorem}\label{thm9}
   Let $U_h^m$ be the solution of the fully discrete Crank-Nicolson Galerkin scheme \eqref{eq6.25}. Then the following stability result holds
      \begin{equation}\label{CN:stab}
      \|U_h^m\|^2 \leq Q \left(\|U_h^0\|^2 + \tau \|\nabla U_h^0\|^2 + \tau\sum_{n=1}^{m} \|f^{n-1/2}\|^2\right), \quad  1\leq m \leq N.
   \end{equation}
   In addition, under the conditions of Theorem \ref{thm:uttt}, the following error estimate holds
   \begin{equation}\label{CN:er}
      \|u^m-U_h^m\| \leq Q(\tau^2+h^2), \quad 1\leq m \leq N.
   \end{equation}
\end{theorem}
\begin{proof}
We take $\chi=\Lambda_{\lambda}(\hat{U}_h^n)$ in \eqref{eq6.25} and follow the similar procedures in Theorem \ref{thm7} to obtain \eqref{CN:stab}.
We then take $\chi = \Lambda_{\lambda}(\hat{\zeta}^n)$ in \eqref{eq6.26}, use $\|\hat{\eta}^m\|\leq Qh^2\|u\|_{L^{\infty}(H^2)}$, and combine Theorem \ref{thm8} to prove \eqref{CN:er}.
\end{proof}

\section{Numerical simulation}\label{sec6}
We carry out numerical experiments to investigate the convergence behavior of different discretization schemes (i.e. the second-order scheme (\ref{eq6.25}) and the first-order scheme in \cite{ZheLiQiu}) and to substantiate the theoretical findings proved in \S \ref{sec5}. We then numerically investigate the behavior of solutions to model \eqref{eq1.1}-(\ref{eq1.3}) to show its dependence on the parameters. Finally, we show the crossover dynamics of \eqref{eq1.1}-(\ref{eq1.3}) in mechanical vibration to demonstrate the advantage of introducing variable exponent.

\subsection{Convergence test}
We consider the problem \eqref{eq1.1}-(\ref{eq1.3})  with $\Omega=(0,1)$, $\mu = \zeta =1$, $T=1$, $\alpha(t)=1 - \frac{4}{5}t$, $u_0(x)=\sin(\pi x)$ and  $f = 1$.
We fix $M=32$ ($M$ is the degree of freedom in spatial discretization) to test the temporal convergence rates of the fully discrete second-order scheme (\ref{eq6.25}) and the first-order scheme proposed in \cite{ZheLiQiu}, while we fix $N=32$ to test its spatial convergence rates.
We denote the temporal errors and the convergence rate as
\begin{equation*}
 E_{2}(\tau,h) = \sqrt{ h\sum_{j=1}^{J-1} \left|U_j^N(\tau,h)-U_j^{2N}(\tau /2,h)\right|^2 }, \quad {\rm rate}^t = \log_{2} \left(\frac{E_{2}(2\tau,h)}{E_{2}(\tau,h)}\right),
\end{equation*}
and the spatial   errors and the corresponding convergence orders are denoted as follows
\begin{equation*}
 F_{2}(\tau,h) = \sqrt{ h\sum_{j=1}^{J-1}  \left|U_j^N(\tau,h)-U_{2j}^{N}(\tau,h/2)\right|^2}, \quad {\rm rate}^x = \log_{2} \left(\frac{F_{2}(\tau,2h)}{F_{2}(\tau,h)}\right).
\end{equation*}
Numerical results are presented in Table \ref{tab1}, which indicate the second-order accuracy of the scheme (\ref{eq6.25}) in both space and time that is consistent with Theorem \ref{thm9}. Furthermore, the numerical errors of the second-order scheme (\ref{eq6.25}) in Table \ref{tab1} are much smaller than those of the first-order scheme proposed in \cite{ZheLiQiu}, which demonstrates its advantages.

\begin{table}[h]
    \center 
    \caption{Discrete $L^2$ errors and convergence rates of the schemes.} \label{tab1}
    \resizebox{\textwidth}{!}{
    \begin{tabular}{cccccccccccc}
      \toprule
      & \multicolumn{2}{c}{Scheme (\ref{eq6.25})} & &\multicolumn{2}{c}{Scheme (\ref{eq6.25})} & &\multicolumn{2}{c}{Scheme in \cite{ZheLiQiu}} \\
       \cmidrule(lr){2-3}\cmidrule(lr){5-6} \cmidrule(lr){8-9}
       $M$  & $F_{2}(\Delta t,h)$ & ${\rm rate}^x$ & $N$  & $E_{2}(\Delta t,h)$ & ${\rm rate}^t$ & $N$  & $E_{2}(\Delta t,h)$ & ${\rm rate}^t$\\
      \midrule
        32  & $3.5833 \times 10^{-5}$   &    *    &  64   & $7.1473 \times 10^{-6}$  &  *    &  64   & $2.6569 \times 10^{-4}$  &  * \\
        64  & $9.0121 \times 10^{-6}$   &  1.99   &  128  & $1.7857 \times 10^{-6}$  &  2.00 &  128  & $1.3529 \times 10^{-4}$  &  0.97\\
        128 & $2.2589 \times 10^{-6}$   &  2.00   &  256  & $4.4796 \times 10^{-7}$  &  2.00 &  256  & $6.8202 \times 10^{-5}$  &  0.99\\
        256 & $5.6559 \times 10^{-7}$   &  2.00   &  512  & $1.1239 \times 10^{-7}$  &  1.99 &  512  & $3.4232 \times 10^{-5}$  &  0.99\\
        512 & $1.4153 \times 10^{-7}$   &  2.00   &  1024 & $2.8200 \times 10^{-8}$  &  1.99 &  1024 & $1.7147 \times 10^{-5}$  &  1.00\\
        \text{Theory} &     &  2.00  &          &      &    2.00  &  &      &    1.00  \\
      \bottomrule
    \end{tabular}
    }
\end{table}

\subsection{Solutions under different parameters} \label{S1:Nume}

We investigate the solutions to problem \eqref{eq1.1}-(\ref{eq1.3})  and its dependence on the parameters.
Let $\Omega = (0, 1)$, $[0, T] = [0, 10]$, $u_0= 0$, $\mu = 0.1$, $\zeta =1$, and $f = e^{-(t+\frac{(x-0.5)^2}{2})}$, which is roughly an initial point source located at the center $x=0.5$ of $\Omega$. To test the effects of the exponent,  we select $\alpha(t)\equiv\alpha$ for different $0<\alpha<1$ and compute $u(0.5,t)$ with $M = 128$ and $N=1024$ in  Fig.~\ref{Fig_u1}(left), which indicates that: (i) For the case $\alpha =1$, the solution exhibits the wave propagation behavior with decreasing amplitude over time \cite{Cheung,Chung}, which corresponds to the fact that the kernel $k$ becomes the identity such that the governing equation \eqref{eq1.1}  could be reformulated as a wave equation with a damping term; (ii) For the case $\alpha < 1$, the wave behavior is weakened due to the viscoelastic property of the governing equation.

 To investigate the effects of the viscosity $\mu$, we present the solution curves in Fig.~\ref{Fig_u2}(right) under the same data and $\alpha(t) = 1- \f{t}{20}$, from which we observe that the maximum amplitudes of the solutions decrease as $\mu$ increases, which reveals the greater dissipation effects under the larger viscosity $\mu$.

\begin{figure}[h]
\setlength{\abovecaptionskip}{0pt}
\centering
\includegraphics[width=2.5in,height=2in]{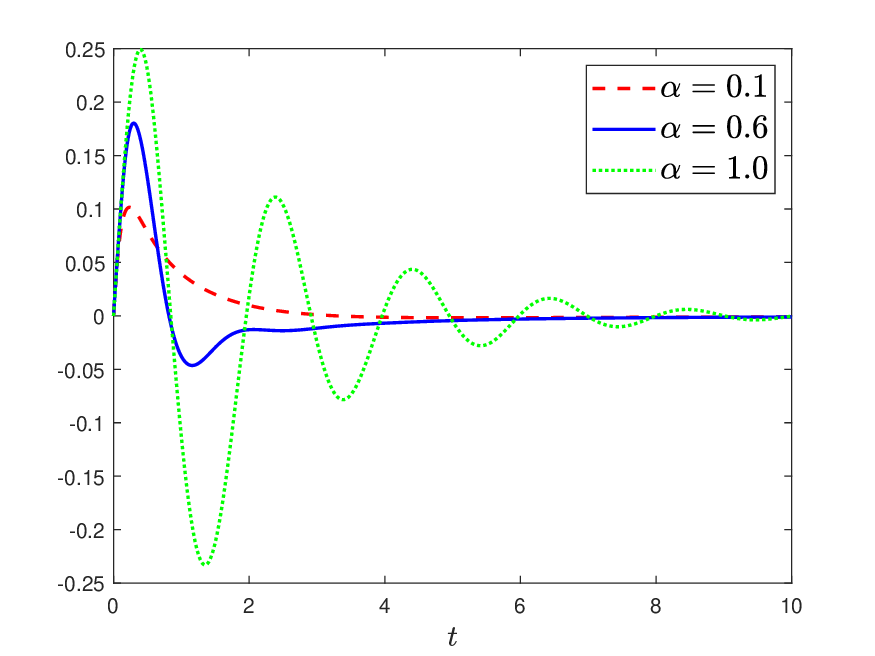}
\includegraphics[width=2.5in,height=2in]{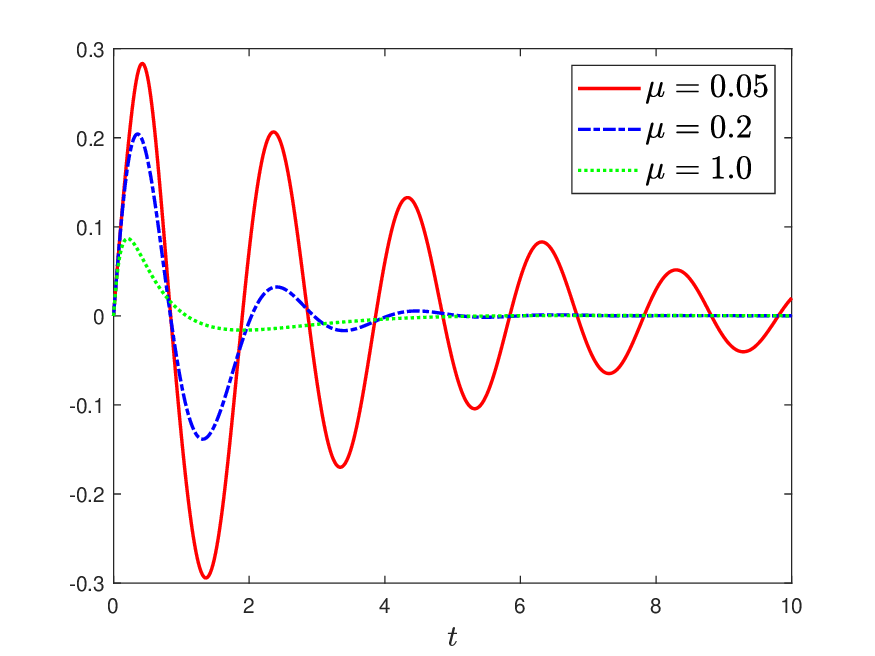}

\caption{Solution curves of $u(0.5,t)$ under different parameters.}
\label{Fig_u1}
\end{figure}

\subsection{Crossover dynamics in mechanical vibration}
We show the crossover dynamics of model (\ref{eq1.1})-(\ref{eq1.3}) in mechanical vibration.  Let $\Omega = (0,10)$, $[0, T]= [0, 150]$, $\mu=0.4$, $\zeta=0.05$, $u_0=0$, $f = e^{-[\frac{t}{2}+\frac{(x-5)^2}{8}]}$, $M=128$ and $N=512$. We plot the solution curves of $u(5,t)$ under different kernels in (\ref{kzz1}) in Fig.~\ref{Fig_u2}, from which we observe that the solution under the multiscale kernel $k$ coincides with that under the quasi-exponential kernel $k_0$ within a relatively short temporal interval, and then gradually captures the long-term viscoelastic behavior modeled by the equation (\ref{eq1.1}) with the power-law kernel $k_\infty$. Such transition indicates the capability of the multiscale kernel $k$ in modeling the crossover dynamics due to, e.g. the change of material properties caused by fatigue and deformation.

\begin{figure}[h]
\setlength{\abovecaptionskip}{0pt}
\centering
\includegraphics[width=5in,height=2in]{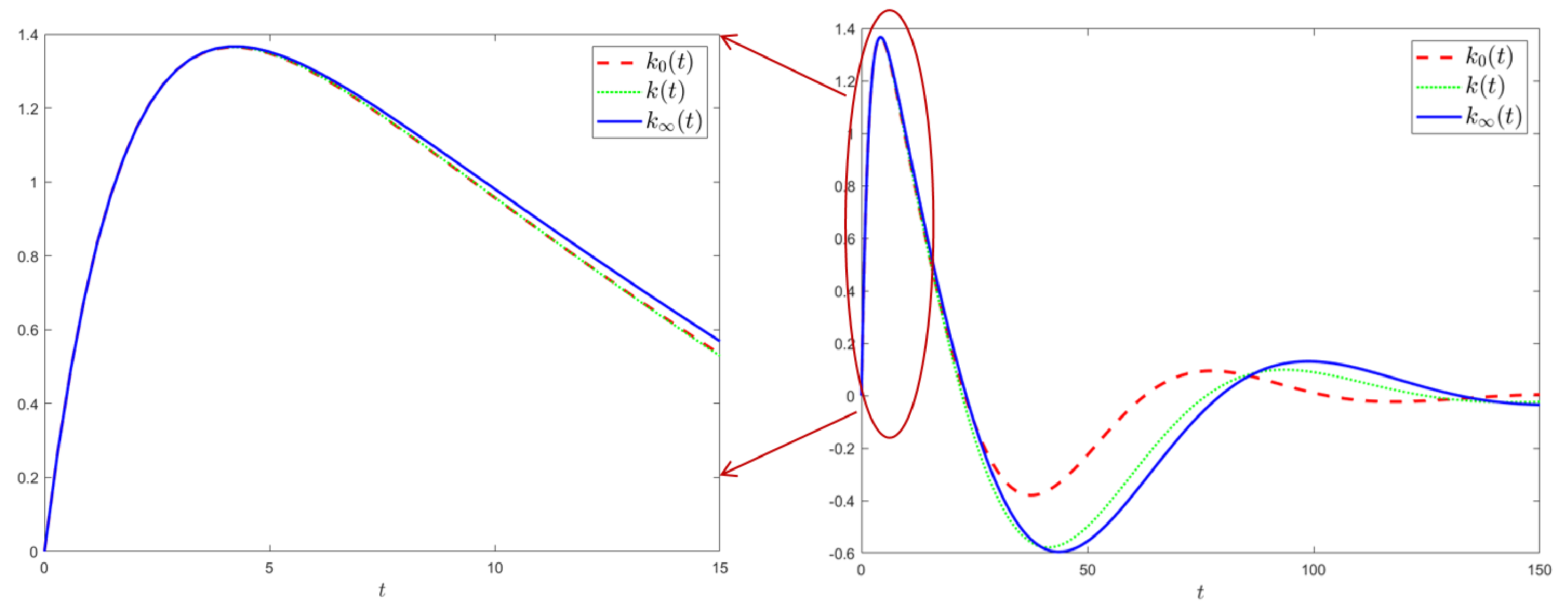}

\caption{Solution curves of $u(5,t)$ under different kernels.}
\label{Fig_u2}
\end{figure}

\section*{Acknowledgments}
This work was partially supported by the National Natural Science Foundation of China (12301555, 12271303), the National Key R\&D Program of China (2023YFA1008903), the Taishan Scholars Program of Shandong Province (tsqn202306083), the Natural Science Foundation of Shan-dong Province for Outstanding Youth Scholars (ZR2024JQ008), the Major Fundamental Research Project of Shandong Province
of China (ZR2023ZD33), and the Postdoctoral Fellowship Program of CPSF (GZC20240938).

\end{document}